\documentclass[11pt]{amsart}
\usepackage{amssymb,graphicx}

\newtheorem{theorem}{Theorem}[section]
\newtheorem{corollary}[theorem]{Corollary}
\newtheorem{lemma}[theorem]{Lemma}
\newtheorem{proposition}[theorem]{Proposition}
\theoremstyle{remark}
\newtheorem{definition}[theorem]{Definition}
\newtheorem{example}[theorem]{Example}
\newtheorem{remark}[theorem]{Remark}
\newtheorem{question}[theorem]{Question}
\numberwithin{equation}{section}

\newcommand{\C}{\mathbb{C}}
\newcommand{\DD}{\mathbb{D}}
\newcommand{\R}{\mathbb{R}}
\newcommand{\T}{\mathbb{T}}
\newcommand{\abs}[1]{\lvert#1\rvert}
\newcommand{\norm}[1]{\lVert#1\rVert}
\DeclareMathOperator{\cp}{cap}
\DeclareMathOperator{\dist}{dist}

\DeclareMathOperator{\im}{Im}
\DeclareMathOperator{\diam}{diam}
\DeclareMathOperator{\sgn}{sign}

\title{Symmetrization and extension of planar bi-Lipschitz maps}

\author{Leonid V. Kovalev}
\address{215 Carnegie, Mathematics Department, Syracuse University, Syracuse, NY 13244, USA}
\email{lvkovale@syr.edu}
\thanks{Supported by the National Science Foundation grant DMS-1362453.}

\subjclass[2010]{Primary 26B35; Secondary 30C35, 31A15}
\keywords{Bi-Lipschitz extension, conformal map, harmonic measure}

\begin{document}

\begin{abstract} We show that every centrally symmetric bi-Lipschitz embedding of the circle into the plane
can be extended to a global bi-Lipschitz map of the plane with linear bounds on the distortion. This answers a 
question of Daneri and Pratelli in the special case of centrally symmetric maps. For general bi-Lipschitz embeddings our distortion bound has a combination of linear and cubic growth, which improves on the prior results. The proof involves a symmetrization result for bi-Lipschitz maps which may be of independent interest. 
\end{abstract}

\maketitle

\section{Introduction}

A map $f\colon X\to Y$ between subsets of Euclidean spaces is called a bi-Lipschitz embedding if there exist positive constants $L$ and $\ell$ such that 
\begin{equation}\label{BLdef}
\ell\abs{a-b}\le \abs{f(a)-f(b)}\le L\abs{a-b}, \quad a,b\in X.
\end{equation}
To emphasize the role of constants, $f$ may be called  $(L, \ell)$-bi-Lipschitz. The lower bound $\ell$ is often taken to be $1/L$ in the literature but for our purpose keeping track of two  constants separately is more natural.  

The Lipschitz Schoenflies theorem was proved by Tukia in~\cite{Tu, Tu2} (see also~\cite{JK,Lat2}). It asserts that every bi-Lipschitz embedding $f\colon \T\to\C$, where $\T$ is the unit circle, can be extended to a global bi-Lipschitz map $F$. In its original form this theorem was not quantitative in that it did not provide Lipschitz constants $(L', \ell')$ for the extension.      

Daneri and Pratelli~\cite{DP} obtained a quantitative bi-Lipschitz extension theorem: an $(L, 1/L)$-bi-Lipschitz embedding $f\colon \T\to\C$ has a $(CL^4, 1/(CL^4))$-bi-Lipschitz extension with a universal constant $C$. They asked whether linear control of the distortion constants is possible, which is a natural question considering that linear distortion bounds are standard for Lipschitz extension problems~\cite{BB}. The following theorem provides such a result for the upper Lipschitz constant $L'$, while the lower constant $\ell'$ is cubic (which is still an improvement on the aforementioned $4$th degree estimate). 

\begin{theorem}\label{mainthmgen} Every $(L,\ell)$-bi-Lipschitz embedding $f\colon \T\to \C$ can be extended to an $(L', \ell')$-bi-Lipschitz automorphism $F\colon \C\to\C$  with $L'= 10^{28} L$ and $\ell' = 10^{-25} \ell^2/L$. 
\end{theorem}

If the distortion is measured by the ratio of upper and lower Lipschitz constants, then Theorem~\ref{mainthmgen} provides a quadratic bound, namely $L'/\ell' \le 10^{53} (L/\ell)^2$.

A map $f$ is \textit{centrally symmetric} if $f(-z) = -f(z)$ for all $z$ in the domain of $f$. For such maps we can settle the problem completely. 

\begin{theorem}\label{mainthm} A centrally symmetric $(L,\ell)$-bi-Lipschitz embedding $f\colon \T\to \C$ can be extended to a centrally symmetric $(L', \ell')$-bi-Lipschitz automorphism $F\colon \C\to\C$  with $L'= 10^{27} L$ and $\ell' = 10^{-23} \ell$. 
\end{theorem}

Recent papers on quantitative bi-Lipschitz extension include~\cite{ATV, ATV2, AV, AS, Ka, Ko, Mac, T}. In particular, Alestalo and V\"ais\"al\"a~\cite{AV} observed that the bi-Lipschitz form of the Klee trick incurs quadratic distortion growth; it is unclear whether this can be made linear. In~\cite{Ko} the author proved that bi-Lipschitz extension with linear distortion bounds is possible for maps $f\colon \R\to\C$. It should be noted that while conjugation by a M\"obius map appears to reduce the extension problem for $\T$ to the same problem for $\R$, this is not so when linear distortion bounds are desired. Conjugating, extending, and conjugating back yields nonlinear bounds such as $CL^9$. 

The paper is structured as follows. 
Section~\ref{harmestsec} relates harmonic measure to metric properties of sets, following the Beurling-Nevanlinna theorem. Section~\ref{confJordansec} uses harmonic measure to estimate the derivative of a  conformal map that will be used in the extension process. In \S\ref{BAextsec} we study the properties of an extension of a circle homeomorphism obtained by the Beurling-Ahlfors method~\cite{BA}. It would be interesting to employ the conformally natural Douady-Earle extension~\cite{DE} instead, but its nonlocal nature presents an obstacle. 

In \S\ref{disksec} the aforementioned results are combined to produce an extension of a centrally symmetric  embedding of $\T$ to a bi-Lipschitz map of the unit disk $\DD$. The exterior domain $\C\setminus \overline{\DD}$ requires a separate treatment; as noted above, M\"obius conjugation is not an option for us. 
The required estimates for harmonic measure and conformal maps of $\C\setminus \overline{\DD}$  are  obtained in~\S\ref{extsec}, and this allows the proof of Theorem~\ref{mainthm} to be completed in~\S\ref{globalsec}. To derive Theorem~\ref{mainthmgen} from Theorem~\ref{mainthm}, we develop a symmetrization process for bi-Lipschitz maps in~\S\ref{symmsec}. The paper concludes with~\S\ref{concludesec} presenting the proof of Theorem~\ref{mainthmgen} and some open questions.  

Throughout the paper, $D(a,r)$ is the open disk of radius $r$ with center $a$. As a special case, $\DD = D(0,1)$ is the unit disk, and $\T = \partial \DD$ is the unit circle. 

\section{Harmonic measure estimates}\label{harmestsec}

Our starting point is a classical harmonic measure estimate~\cite[Corollary 4.5.9]{Ranb} which is a consequence of
the Beurling-Nevanlinna projection theorem.

\begin{proposition}\label{BNprop} Let $\Omega\subset \C\setminus \{0\}$ be a simply connected domain. Pick a point
$\zeta\in\Omega$ and let $\rho>0$.

(a) If $\abs{\zeta}<\rho$, then
\begin{equation}\label{BN1}
\omega(\zeta,\partial \Omega\cap D(0,\rho), \Omega)\ge \frac{2}{\pi}\sin^{-1}\left(\frac{\rho-\abs{\zeta}}{\rho+\abs{\zeta}}\right).
\end{equation}

(b) If $\abs{\zeta}>\rho$, then
\begin{equation}\label{BN2}
\omega(\zeta,\partial \Omega\cap \overline{D}(0,\rho), \Omega)\le \frac{2}{\pi}\cos^{-1}\left(\frac{\abs{\zeta}-\rho}{\abs{\zeta}+\rho}\right).
\end{equation}
\end{proposition}

We need the following corollary of Proposition~\ref{BNprop}.

\begin{corollary}\label{BNcor} Let $\Omega\subset \C$ be a simply connected domain. Consider a point $\zeta \in \Omega$ and a subset $\Gamma\subset \partial \Omega$. Suppose that $\omega(\zeta,\Gamma,\Omega)\ge \epsilon>0$. Then 
\begin{equation}\label{BNcor1}
\dist(\zeta,\Gamma) \le \csc^2 \left(\frac{\pi \epsilon}{4}\right) \dist(\zeta,\partial\Omega);
\end{equation}
\begin{equation}\label{BNcor2}
\diam \Gamma \ge \tan^2 \left(\frac{\pi \epsilon}{4} \right) \dist(\zeta,\Gamma).
\end{equation}
\end{corollary}

\begin{proof} To prove~\eqref{BNcor1}, translate $\Omega$ so that $0$ is a nearest boundary point to $\zeta$, that is $0\in\partial\Omega$ and $\dist(z,\Omega) = \abs{\zeta}$. Let $\rho = \dist(\zeta,\Gamma) - \abs{\zeta}$. If $\rho \le \abs{\zeta}$, then $\dist(\zeta,\Gamma) \le 2 \dist(z,\Omega)$ and so ~\eqref{BNcor1} holds. Otherwise,~\eqref{BN1} implies 
\begin{equation}\label{bnc1a}
\frac{2}{\pi}\sin^{-1}\left(\frac{\rho-\abs{\zeta}}{\rho+\abs{\zeta}}\right)
\le \omega(\zeta,\partial \Omega\cap D(0,\rho), \Omega) \le 1-\epsilon 
\end{equation}
where the second inequality holds because $D(0,\rho)$ is disjoint from $\Gamma$. 
Rearranging~\eqref{bnc1a} yields $\rho/\abs{\zeta} \le \cot^2(\pi \epsilon/4)$, which implies ~\eqref{BNcor1}.

To prove~\eqref{BNcor2}, translate $\Omega$ so that $0\in \Gamma$. Let $\rho=\diam\Gamma$ and note that $\Gamma\subset \overline{D}(0,\rho)$. If $\abs{\zeta} \le \rho$, then $\dist(\zeta,\Gamma)\le \diam \Gamma$ and so~\eqref{BNcor2} holds. Otherwise, by~\eqref{BN2},
\begin{equation*}
\frac{2}{\pi}\cos^{-1}\left(\frac{\abs{\zeta}-\rho}{\abs{\zeta}+\rho}\right)
\ge \omega(\zeta,\partial \Omega\cap D(0,\rho), \Omega) \ge \epsilon 
\end{equation*}
which implies $\abs{\zeta}/\rho \le  \cot^2(\pi \epsilon/4)$, i.e., \eqref{BNcor2}. 
\end{proof}

Given a point $z = re^{i\theta}$ with $e^{-2\pi}<r<1$, let $\delta = \log (1/r)$ and introduce four arcs of the unit circle $\T$: 
\begin{equation}\label{gammacurves}
\begin{split}
&\gamma_1 = \{e^{it}\colon \theta-2\delta \le  t\le \theta - \delta \} \\
&\gamma_2 = \{e^{it}\colon \theta-\delta \le  t\le \theta - \delta/2 \} \\
&\gamma_3 = \{e^{it}\colon \theta+\delta/2 \le  t\le \theta + \delta \} \\
&\gamma_4 = \{e^{it}\colon \theta+\delta \le  t\le \theta + 2\delta\}
\end{split}
\end{equation}

Since the length of each arc is comparable to its distance from $z$, one expects its harmonic measure with respect to $z$ to be bounded below by a positive constant. The following lemma makes this explicit. The constraints on $\abs{z}$ in~\eqref{lowerharm1} and ~\eqref{lowerharm2} are imposed so that the arcs involved are contained in a semicircle, which will be important later. 

\begin{lemma}\label{lowerharm} Using notation~\eqref{gammacurves} for $j=1,\dots,4$, we have
\begin{equation}\label{lowerharm1}
\omega(z, \gamma_j, \DD) \ge \frac{1}{30\pi} \quad \text{if $j=1,4$ and $e^{-\pi/4}<|z|<1$}
\end{equation}
and
\begin{equation}\label{lowerharm2}
\omega(z, \gamma_j, \DD) \ge \frac{1}{64\pi} \quad \text{if $j=2,3$ and $e^{-2\pi}<|z|<1$}
\end{equation}
\end{lemma}
\begin{proof} Since the logarithmic function is concave, the function $x\mapsto \log x/(x-1)$ is decreasing for $x>1$. Therefore, 
\begin{equation}\label{logbounds1}
1 < \frac{\log(1/r)}{1-r} < \frac{\pi/4}{1-e^{-\pi/4}} < 2  \quad \text{for }\ e^{-\pi/4} < r< 1.
\end{equation}
and 
\begin{equation}\label{logbounds2}
1 < \frac{\log(1/r)}{1-r} < \frac{2\pi}{1-e^{-2\pi}} < 7  \quad \text{for }\ e^{-2\pi} < r< 1.
\end{equation}
For $\zeta\in \gamma_j$, $j=1,4$, the triangle inequality and ~\eqref{logbounds1} imply
\[
\abs{\zeta-z}^2 \le |1-r+2\delta| |1-r+2\delta| < 5 (1-r) (3\delta) = 15(1-r)\delta  
\]
This leads to Poisson kernel estimates, using the explicit form of the kernel~\cite[Theorem I.1.3]{GMb}:
\[
P_z(\zeta) = \frac{1}{2\pi}\frac{1-r^2}{\abs{\zeta-z}^2} 
\ge \frac{1}{2\pi}\frac{1-r}{15(1-r)\delta} = \frac{1}{30\pi \delta}
\]
Since the length of $\gamma_j$ is $\delta$, inequality~\eqref{lowerharm1} follows.

Similarly, for $j=2,3$ we use \eqref{logbounds2} to obtain 
\[
\abs{\zeta-z}^2 \le |1-r+\delta| |1-r+\delta| < 8 (1-r) (2\delta) = 16(1-r)\delta  
\]
hence $P_z(\zeta) \ge 1/(32\pi \delta)$. Since the length of  $\gamma_j$ is $\delta/2$, inequality~\eqref{lowerharm2} follows.
\end{proof}

\section{Conformal map onto a Jordan domain}\label{confJordansec}
 
The harmonic measure estimates  in~\S\ref{harmestsec} allow us to control the derivative of a conformal map in terms of the images of boundary arcs $\gamma_j$ introduced in~\eqref{gammacurves}.
 
\begin{lemma}\label{intconf}
 Let $\Omega\subsetneq \C$ be a Jordan domain. Fix a conformal map $\Phi $  of $\DD$ onto $\Omega$ and consider a point $z = re^{i\theta}$ with $e^{-2\pi}<r<1$. Referring to notation~\eqref{gammacurves}, let $\Gamma_j\subset\partial\Omega$ be the image of $\gamma_j$ under the boundary map induced by $\Phi$. Also let $\zeta = \Phi(z)$ and $\rho = \dist(\zeta, \partial\Omega)$.  Then 
\begin{equation}\label{confderest1}
\abs{\Phi'(z)} \ge  
\frac{\dist(\Gamma_1, \Gamma_4)}{60000\log(1/r)},\quad  e^{-\pi/4} < r< 1, 
\end{equation}
and
\begin{equation}\label{confderest2}
\abs{\Phi'(z)} \le 
2\cdot 10^6\,\dfrac{\min(\diam \Gamma_2, \diam \Gamma_3)}{\log(1/r)},\quad  e^{-2\pi} < r< 1.
\end{equation}
\end{lemma}
  
\begin{proof} The Koebe $1/4$ theorem and the Schwarz lemma imply the standard estimate~\cite[Corollary 1.4]{Pomb}
\begin{equation}\label{Koebe}
  \rho \le (1-r^2) |\Phi'(z)| \le 4\rho.
\end{equation}
A concavity argument similar to~\eqref{logbounds1}--\eqref{logbounds2} yields  
\begin{equation}\label{logbounds1a}
\frac12 < \frac{\log(1/r)}{1-r^2} < \frac{\pi/4}{1-e^{-\pi/2}} < 1  \quad \text{for }\ e^{-\pi/4} < r< 1,
\end{equation}
and 
\begin{equation}\label{logbounds2a}
\frac12 < \frac{\log(1/r)}{1-r^2} < \frac{2\pi}{1-e^{-4\pi}} < 7  \quad \text{for }\ e^{-2\pi} < r< 1.
\end{equation}

\textit{Proof of ~\eqref{confderest1} for $e^{-\pi/4}< r <1$.} 
From ~\eqref{BNcor1} and ~\eqref{lowerharm1} it follows that
\[
\dist(\Gamma_1, \Gamma_4) \le 
\dist(\zeta, \Gamma_1) + \dist(\zeta, \Gamma_4) \le 
2\csc^2 \left(\frac{1}{120}\right)\rho  \le 30000 \rho
\]
which in view of~\eqref{Koebe} and~\eqref{logbounds1a} implies 
\begin{equation*}
\abs{\Phi'(z)} \ge \frac{\rho}{1-r^2} \ge 
\frac{\dist(\Gamma_1, \Gamma_4)}{30000(1-r^2)} \ge \frac{\dist(\Gamma_1, \Gamma_4)}{60000\log(1/r)}
\end{equation*}
proving~\eqref{confderest1}.

\textit{Proof of ~\eqref{confderest2} for $e^{-2\pi}< r <1$.} From ~\eqref{BNcor2} and ~\eqref{lowerharm2} it follows that
\[
\min(\diam \Gamma_2, \diam \Gamma_3) \ge  \tan^2 \left(\frac{1}{256}\right)\rho 
\]
hence 
\[
\abs{\Phi'(z)} \le \frac{4\rho}{1-r^2} \le 
 \frac{4\min(\diam \Gamma_2, \diam \Gamma_3)}{\tan^2 \left(\frac{1}{256}\right)(1-r^2)}
< 2\cdot 10^6 \frac{\min(\diam \Gamma_2, \diam \Gamma_3)}{\log(1/r)}
\]
where the last step uses ~\eqref{logbounds2a}.
\end{proof}

\section{Extension of a circle homeomorphism}\label{BAextsec}

A key element of the proof, going back to Tukia~\cite{Tu2}, is pre-composing  a conformal map with a disk homeomorphism obtained by extending a suitable circle homeomorphism. This extension is carried out below. Lemma~\ref{homeodist} is where the assumption of central symmetry is crucial: it ensures that the  image  of any set contained in a semicircle is also contained in a semicircle, enabling the comparison of intrinsic and extrinsic distances on $\T$.  

A sense-preserving circle homeomorphism $\psi\colon \T\to\T$ lifts to an increasing homeomorphism $\chi$ of the real line onto itself, which satisfies $\psi (e^{it}) = e^{i \chi(t)}$ for all $t\in \R$, and $\chi(t+2\pi) = \chi(t) + 2\pi$.  As a consequence, 
\begin{equation}\label{notfar}
\abs{\chi(t)-t-\chi(0)} \le 2\pi, \quad t\in\R. 
\end{equation}

Let $\chi_e$ denote the following variant of the Beurling-Ahlfors extension of $\chi$: 
\begin{equation}\label{ext1}
\chi_e(x+iy)=\frac{1}{2}\int_{-1}^{1} \chi(x+ty)(1+2i\sgn t)\,dt.
\end{equation}
This is a diffeomorphism of the upper halfplane onto itself~\cite{Ahb,BA}. It differs from the map considered in ~\cite{Ahb} only by the factor of $2$ in front of the imaginary part. The contribution of this factor is that the derivative matrix $D\chi_e$ is multiplied by $\left(\begin{smallmatrix} 1 & 0 \\ 0 & 2 \end{smallmatrix}\right)$ on the left. Due to the submultiplicativity of operator norm, the inequalities (4.9) and (4.11) from~\cite{Ko} still apply to this variant of the extension, with an extra factor of $2$:
\begin{equation}\label{maxstr}
\norm{D\chi_e(x+iy)} \le 2\,\frac{\chi(x+y)-\chi(x-y)}{y}
\end{equation}
\begin{equation}\label{minstr}
\norm{D\chi_e(x+iy)^{-1}} \le \frac{4y}{\min(\chi(x+y)-\chi(x+\frac{y}2), \chi(x-\frac{y}2)-\chi(x-y))}
\end{equation}

The reason for inserting $2$ in front of $i\sgn t$ in~\eqref{ext1} is the following estimate, which employs~\eqref{notfar}; it asserts that $\chi_e$ maps each horizontal line onto a curve with a bounded distance from the line. 
\begin{equation}\label{imext}
\begin{split}
\abs{\im \chi_e(x+iy) - y} & = \left|\int_{-1}^{1} (\chi(x+ty) - ty) \sgn t \,dt\right|    \\
&= \left | \int_{-1}^{1} (\chi(x+ty) - (x+ty + \chi(0))) \sgn t \,dt \right|\\
& \le 4\pi  
\end{split}
\end{equation}

Since $\chi$ commutes with translation by $2\pi$, so does $\chi_e$: that is, $\chi_e(z+2\pi)=\chi_e(z)+2\pi$. This allows us to define a map $\Psi$ of the unit disk $\DD$ onto itself as follows.
\begin{equation}\label{halfplane2disk}
\Psi(e^{iz}) = \exp(i \chi_e(z)),\quad  \Psi(0)=0
\end{equation}
This is a diffeomorphism of the punctured disk $\DD\setminus \{0\}$ onto itself, and also a homeomorphism of $\DD$ onto $\DD$. According to~\eqref{imext}, 
\begin{equation}\label{modext}
e^{-4\pi}\abs{\zeta}\le \abs{\Psi(\zeta)} \le e^{4\pi}\abs{\zeta},\quad \zeta\in\DD.
\end{equation}
Using the chain rule and~\eqref{modext}, we obtain 
\begin{equation}\label{chainPsi}
\norm{D\Psi(e^{iz})} \le e^{4\pi} \norm{ D\chi_e(z)},\qquad 
\norm{D\Psi(e^{iz})^{-1} } \le e^{4\pi} \norm{ D\chi_e(z)^{-1}}.
\end{equation}

 \begin{lemma}\label{homeodist} Let $\psi\colon \T\to\T$ be a sense-preserving circle homeomorphism such that 
\begin{equation}\label{symm1}
\psi(-z) = -\psi(z),\quad z\in\T
\end{equation}
Let $\Psi\colon \DD\to\DD$ be the extension of $\psi$ defined by~\eqref{halfplane2disk}. Fix a point $z = re^{i\theta}$ with $0<r<1$. Referring to notation~\eqref{gammacurves}, let $\sigma_j\subset\T$ be the image of $\gamma_j$ under $\psi$.  Then the derivative matrix $D\Psi(z)$ satisfies
\begin{equation}\label{BAderest1}
\norm{D\Psi(z)} \le  
\begin{cases}
e^{4\pi} \pi \dfrac{\dist(\sigma_1, \sigma_4)}{\log(1/r)},\quad 
&e^{-\pi/4} < r< 1; \\ 
20 e^{4\pi} ,\quad &0<r \le e^{-\pi/4}
\end{cases}
\end{equation}
and
\begin{equation}\label{BAderest2}
\norm{D\Psi(z)^{-1}} \le 
\begin{cases}
\dfrac{4e^{4\pi} \log(1/r)}{\min(\diam \sigma_2, \diam \sigma_3)},\quad &e^{-2\pi} < r< 1; \\ 
16 e^{4\pi} ,\quad &0<r \le e^{-2\pi};
\end{cases}
\end{equation}
\end{lemma} 

\begin{proof} Note that $\Psi(re^{i\theta}) = \exp(i \chi_e(\theta + i\delta))$ where $\delta = \log(1/r)$. The central symmetry property ~\eqref{symm1} implies that the lifted homeomorphism $\chi\colon\R\to\R$ satisfies 
\begin{equation}\label{piprop}
\chi(t+\pi) = \chi(t)+\pi
\end{equation}
and the same holds for its extension $\chi_e$. Consequently, $\Psi$ inherits the central symmetry.   

\textit{Proof of~\eqref{BAderest1}.} In view of~\eqref{chainPsi}, the estimate ~\eqref{maxstr} yields
\begin{equation}\label{dPsi1}
\norm{ D\Psi(z)}  \le 2e^{4\pi}\,  \frac{\chi(\theta+\delta)- \chi(\theta-\delta)}{\delta}.
\end{equation}
Suppose $r>e^{-\pi/4}$. Then the union $\gamma_1\cup \gamma_4$ is contained in a semicircle. Since $\psi$ is centrally symmetric, it maps a semicircle to another semicircle. Within a semicircle, Euclidean distance is comparable to arcwise distance. Specifically, 
\begin{equation}\label{arcdist}
\dist(\sigma_1, \sigma_4) \ge \frac{2}{\pi} (\chi(\theta+\delta)- \chi(\theta-\delta)). 
\end{equation}
From ~\eqref{dPsi1} and~\eqref{arcdist}, the inequality~\eqref{BAderest1} follows. 

Now consider the case $0<r\le e^{-\pi/4}$. By virtue of ~\eqref{piprop}, 
\[
\chi(\theta+\delta)- \chi(\theta-\delta) \le \pi\left(
\left \lfloor \frac{2\delta}{\pi}\right\rfloor+2\right). 
\]
Since $2\le 8\delta/\pi$, it follows that 
\[
\chi(\theta+\delta)- \chi(\theta-\delta) \le \pi\left(
 \frac{2\delta}{\pi}  + \frac{8\delta}{\pi} \right) = 10\delta.
\]
Returning to~\eqref{dPsi1}, we get  $\norm{ D\Psi(z)}  \le 20e^{4\pi}$ in this case.

\textit{Proof of~\eqref{BAderest2}.}
In view of~\eqref{chainPsi}, the estimate ~\eqref{minstr} yields
\begin{equation}\label{dPsi2}
\norm{D\Psi(z)^{-1}}   \le \frac{4 e^{4\pi} \delta }{\min(\chi(\theta+\delta)-\chi(\theta+\delta/2), \chi(\theta-\delta/2)-\chi(\theta-\delta))}.
\end{equation}
First suppose $r>e^{-2\pi}$. The denominator of ~\eqref{dPsi2} is the length of the shorter of the arcs $\sigma_2, \sigma_3$. Hence it is bounded from below by the minimum of $\diam \sigma_2$ and  $\diam\sigma_3$, which yields ~\eqref{BAderest2}.

Now consider the case $0<r\le e^{-2\pi}$. Then $\delta/2\ge \pi$, which by ~\eqref{piprop} implies
\[
\chi(\theta+\delta)-\chi(\theta+\delta/2) \ge \pi \left\lfloor \frac{\delta/2}{\pi}\right \rfloor 
\ge \frac{\delta}{4}.
\]
The same bound holds for $\chi(\theta-\delta/2)-\chi(\theta-\delta)$, which shows that the right hand side of ~\eqref{dPsi2} is bounded above by $16e^{4\pi}$ and thus completes the proof of ~\eqref{BAderest2}. 
\end{proof}

\section{Bi-Lipschitz extension in the unit disk}\label{disksec}

In this section we prove a half of Theorem~\ref{mainthm}, constructing an extension of $f$ in the unit disk $\DD$.

\begin{theorem}\label{mainthmInt} Any centrally symmetric $(L,\ell)$-bi-Lipschitz embedding $f\colon \T\to \C$ can be extended to a centrally symmetric embedding $F\colon \overline{\DD}\to\C$  such that $F$ is differentiable in $\DD\setminus\{0\}$ and its derivative matrix $DF$ satisfies $\norm{DF}\le 10^{13} L$ and $\norm{DF^{-1}} \le 10^{11} / \ell$ in $\DD\setminus\{0\}$.
\end{theorem}

\begin{proof} There is no loss of generality in assuming $f$ is sense-preserving; that is, the Jordan curve $f(\T)$ is traversed counter-clockwise. This curve divides the plane in two domains, one of which, denoted $\Omega$, is bounded and contains $0$. Note that 
\begin{equation}\label{twodisks}
B(0, \ell)\subset  \Omega \subset B(0,L)
\end{equation}
because the quantity 
\[
\abs{f(z)} = \frac{\abs{f(z)-f(-z)}}{\abs{z-(-z)}},\quad z\in\T 
\]
is bounded between $\ell$ and $L$. 

Let $\Phi$ a conformal map of $\DD$ onto $\Omega$ such that $\Phi(0)=0$. Note that $\Phi(-z)=-\Phi(z)$ by the uniqueness of such a map (up to rotation of the domain $\DD$). The inclusion~\eqref{twodisks} implies $\ell \le \abs{\Phi'(0)} \le L$ by the Schwarz lemma. The distortion theorem~\cite[Theorem 2.5]{Durb} states that
\begin{equation}\label{Koebedist1}
\ell\,\frac{1-\abs{z}}{(1+\abs{z})^3}  \le \abs{\Phi'(z)} \le L\,\frac{1+\abs{z}}{(1-\abs{z})^3}, \quad z\in\DD.
\end{equation}
By Carath\'eodory's theorem, $\Phi$ extends to a homeomorphism between $\overline{\DD}$ and $\overline{\Omega}$. Let $\phi\colon \T \to \partial\Omega $ be the induced boundary map. 

Define $\psi\colon \T\to\T$ by $\psi = f^{-1}\circ \phi$. This is a  sense-preserving circle homeomorphism, which is centrally symmetric because $f$ and $\phi$ are. Lemma~\ref{homeodist} provides its extension $\Psi$ to the unit disk. For $e^{-\pi/4}<\abs{z}<1$, the estimates~\eqref{confderest1}  and ~\eqref{BAderest1} yield 
\begin{equation}\label{PsiPhi1}
\norm{D\Psi(z)} \le 60000 e^{4\pi} \pi \frac{\dist(\sigma_1,\sigma_4)}{\dist(\Gamma_1, \Gamma_4)} \abs{\Phi'(z)}. 
\end{equation}
Note that $\Gamma_j = f(\sigma_j)$ for $j=1,\dots,4$ because $f= \phi\circ \psi^{-1}$. 
The bi-Lipschitz property of $f$ will be used here in the form
\begin{equation}\label{BLpropused}
\diam \Gamma_j \le L\diam \sigma_j \quad \text{and} \quad \dist(\Gamma_j, \Gamma_k) \ge \ell \dist(\sigma_j, \sigma_k). 
\end{equation}
Hence~\eqref{PsiPhi1} simplifies to 
\begin{equation}\label{PsiPhi2}
\norm{D\Psi(z)} \le 60000 e^{4\pi} \pi \ell^{-1} \abs{\Phi'(z)} \le 10^{11}\ell^{-1} \abs{\Phi'(z)}. 
\end{equation}
For $0<\abs{z}\le e^{-\pi/4}$, we combine~\eqref{BAderest1} and  ~\eqref{Koebedist1} to obtain
\begin{equation}\label{PsiPhi3}
\norm{D\Psi(z)} \le 20 e^{4\pi} \frac{(1+e^{-\pi/4})^3}{1-e^{-\pi/4}}\ell^{-1} \abs{\Phi'(z)} 
\le   10^{11}\ell^{-1} \abs{\Phi'(z)}. 
\end{equation}

The composition $F = \Phi\circ \Psi^{-1}$ extends $f =\phi\circ \psi^{-1} $. By~\eqref{PsiPhi2}, ~\eqref{PsiPhi3} and the chain rule, $\norm{DF^{-1}}\le 10^{11} / \ell$ in $\DD\setminus\{0\}$.

For $e^{-2\pi}  < \abs{z}<1$ we use ~\eqref{confderest2}  and ~\eqref{BAderest2} to obtain
\begin{equation*}
\norm{D\Psi(z)^{-1}} \le  8e^{4\pi}\cdot 10^6\,\frac{\min(\diam \Gamma_2, \diam \Gamma_3)}{\min(\diam \sigma_2, \diam \sigma_3)} \abs{\Phi'(z)}^{-1} \le 10^{13}L\abs{\Phi'(z)}^{-1}
\end{equation*}
hence $\norm{DF}\le 10^{13}L$. When $0<\abs{z}<e^{-2\pi} $ use ~\eqref{BAderest2} and  ~\eqref{Koebedist1} instead: 
\begin{equation*}
\norm{D\Psi(z)^{-1}} \le  16 e^{4\pi}  \frac{1+e^{-2\pi}}{(1-e^{-2\pi})^3} L \abs{\Phi'(z)}^{-1} 
\end{equation*}
leading to the same conclusion $\norm{DF}\le 10^{13}L$. 
\end{proof}

\section{Harmonic measure and conformal mapping of an exterior domain}\label{extsec}

In this section $\Omega\subset\C$ is a domain such that $\C\setminus \Omega$ is compact,  connected, and contains more than one point. Our goal  is to obtain harmonic measure estimates similar to Corollary~\ref{BNcor} and use them to prove an analog of Lemma~\ref{intconf}. This will be done by applying a suitably chosen M\"obius transformation that maps $\Omega$ onto a simply connected domain in $\C$ minus one point (the image of $\infty$). Since removing one point does not change the harmonic measure, it can be ignored. 

\begin{lemma}\label{unboundedharm} Let $\Omega\subset \C$ be a domain with compact connected complement  containing more than one point. Let  $K=\partial \Omega$ and consider a point $\zeta \in \Omega$ and a subset $\Gamma\subset K$ such that $\omega(\zeta,\Gamma,\Omega)\ge \epsilon>0$. Then 
\begin{equation}\label{distGamma}
\dist(\zeta,\Gamma) \le 4\csc^2 \left(\frac{\pi \epsilon}{4}\right) \dist(\zeta,K);
\end{equation}
\begin{equation}\label{diamGamma}
\diam \Gamma \ge \frac14 \tan^2 \left(\frac{\pi \epsilon}{4} \right) 
\frac{(\diam K - \diam\Gamma)^2}{\diam K(\diam K+\dist(\zeta, \Gamma))} \dist(\zeta,\Gamma).
\end{equation}
\end{lemma}

\begin{proof} In order to prove~\eqref{distGamma}, translate $K$ so that $0\in K$ and $0$ is the point of $K$ that is furthest from $\zeta$. Let $z_1$ be a point of $K$ that is closest to $\zeta$. 
If $\abs{\zeta}< 2\abs{z_1-\zeta}$, then~\eqref{distGamma} holds in the stronger form  
$\dist(\zeta,\Gamma) < 2\dist(\zeta,K)$. So we may assume $\abs{\zeta} \ge 2\abs{z_1-\zeta}$, hence $\abs{z_1} \ge  \abs{\zeta}/2$. 

Under the M\"obius transformation $z\mapsto 1/z$ the sets $\Gamma$ and $K$ are mapped onto sets $\tilde \Gamma$ and $\tilde K$, with the latter certain to be unbounded. Since the harmonic measure is invariant under this transformation, Corollary~\ref{BNcor} yields 
\begin{equation}\label{distGamma1}
\dist(1/\zeta, \tilde\Gamma)\le \csc^2 \left(\frac{\pi \epsilon}{4}\right) \dist(1/\zeta, \tilde K).
\end{equation}
Using the point $z_1 \in K$ chosen above, we get 
\begin{equation}\label{distGamma2}
\dist(1/\zeta, \tilde K) \le \abs{1/\zeta - 1/z_1} = \frac{\abs{\zeta-z_1}}{\abs{\zeta}\abs{z_1}} 
\le  \frac{2\dist(\zeta, K)}{\abs{\zeta}^2}.
\end{equation}
Let $w\in \tilde\Gamma$ be a point realizing the distance $\dist(1/\zeta, \tilde\Gamma)$. Since $1/w \in   K$ and $0$ is the furthest point of $K$ from $\zeta$, it follows that $\abs{1/w} \le 2\abs{\zeta}$. Hence 
\begin{equation}\label{distGamma3}
\dist(\zeta,\Gamma)\le 
\abs{\zeta - 1/w } = \frac{\abs{w-1/\zeta} \abs{\zeta}}{\abs{w}} \le 2 \abs{\zeta}^2 \dist(1/\zeta, \tilde\Gamma). 
\end{equation}
Combining~\eqref{distGamma1}--\eqref{distGamma3} yields~\eqref{distGamma}.

\textit{Proof of~\eqref{diamGamma}}. If $\diam \Gamma=\diam K$ there is nothing to prove. Otherwise, let $\delta =  (\diam K - \diam \Gamma)/2$ and observe that there exists a point $z_1\in K$ such that $\dist(z_1, \Gamma)\ge \delta$. Translate $K$ so that $z_1 = 0$. 

Under the transformation $z\mapsto 1/z$ the sets $\Gamma$ and $K$ are mapped onto sets $\tilde \Gamma$ and $\tilde K$, where $\tilde K$ is unbounded. By Corollary~\ref{BNcor}, 
\begin{equation}\label{diamGamma1}
\diam \tilde \Gamma \ge \tan^2 \left(\frac{\pi \epsilon}{4} \right) \dist(1/\zeta, \tilde \Gamma).
\end{equation}
Here
\begin{equation}\label{diamGamma2}
\diam \tilde \Gamma  = \sup_{a,b\in\Gamma}\frac{\abs{a-b}}{\abs{a}\abs{b}} \le \frac{\diam\Gamma}{\delta^2}.
\end{equation}
Also,
\begin{equation}\label{diamGamma3}
\dist(1/\zeta, \tilde \Gamma) = \inf_{z\in \Gamma}\frac{\abs{z-\zeta}}{\abs{z}\abs{\zeta}} 
\ge \frac{\dist(\zeta, \Gamma)}{\diam K(\diam K +\dist(\zeta,\Gamma))}
\end{equation}
because $\abs{z}\le \diam K$.  Combining~\eqref{diamGamma1}--\eqref{diamGamma3} yields~\eqref{diamGamma}.
\end{proof}

Since the inequality~\eqref{diamGamma} is more involved than its counterpart~\eqref{BNcor2}, we need an  additional estimate in order to use it effectively. 

\begin{corollary}\label{distdiamlem} Under the assumptions of Lemma~\ref{unboundedharm}, let $\Phi\colon \C\setminus \overline{\mathbb D}\to \Omega$ be a conformal map and let $z \in \C\setminus \overline{\mathbb D}$ be the point such that $\Phi(z)=\zeta$. Then 
\begin{equation}\label{diamGammaSimple}
\diam\Gamma \ge \frac{1}{32\abs{z}}\tan^2 \left(\frac{\pi \epsilon}{4} \right) \dist(\zeta,\Gamma).
\end{equation}
\end{corollary}
\begin{proof}  First observe that 
\begin{equation}\label{distdiam}
\dist(\zeta, \Gamma)\le \abs{z}\diam K.
\end{equation}
Indeed, for any $w_0\in \Gamma$ the function $f(w) = (w-w_0)/\Phi^{-1}(w)$ is holomorphic in $\Omega$, bounded at infinity, and bounded by $\diam K$ on the boundary of $\Omega$. Hence $\abs{f(w)}\le \diam K$, which yields~\eqref{distdiam} by letting $w=\zeta$. 

If $\diam\Gamma > \frac12 \diam K$, then~\eqref{diamGammaSimple} holds by virtue of~\eqref{distdiam}. It remains to consider the case $\diam\Gamma \le \frac12 \diam K$. Since 
\[
\frac{(\diam K - \diam\Gamma)^2}{\diam K(\diam K+\dist(\zeta, \Gamma))} 
\ge \frac{(\diam K)^2/4}{\diam K(\diam K+\abs{z}\diam K)} \ge \frac{1}{8\abs{z}} 
\]
the estimate~\eqref{diamGamma} simplifies to~\eqref{diamGammaSimple}. 
\end{proof}

Given a point $z=R e^{i\theta}$ with $1<R<e^{2\pi}$, let $\delta = \log R$ and introduce four arcs $\gamma_1,\dots,\gamma_4\subset\mathbb T$ as in~\eqref{gammacurves}. The conformal invariance of harmonic measure yields an analog of Lemma~\ref{lowerharm} for this situation: 
\begin{equation}\label{lowerharm1e}
\omega(z, \gamma_j, \C\setminus \overline{\DD}) \ge \frac{1}{30\pi} \quad \text{if $j=1,4$ and $1<|z|<e^{\pi/4}$},
\end{equation}
\begin{equation}\label{lowerharm2e}
\omega(z, \gamma_j, \C\setminus \overline{\DD}) \ge \frac{1}{64\pi} \quad \text{if $j=2,3$ and $1<|z|<e^{2\pi}$}.
\end{equation}

We proceed to the main result of the section: distortion estimates for a conformal map of $\C\setminus\overline{\DD}$. 

\begin{lemma}\label{extconf}
 Let $\Omega\subset \C$ is a domain with compact connected boundary  $K=\partial \Omega$. Fix a conformal map $\Phi $  of $\C\setminus\overline{\DD}$ onto $\Omega$ and consider a point $z = Re^{i\theta}$ with $R>1$. Referring to notation~\eqref{gammacurves}, let $\Gamma_j\subset K$ be the image of $\gamma_j$ under the boundary map induced by $\Phi$. Also let $\zeta = \Phi(z)$. Then 
\begin{equation}\label{extderest1}
\abs{\Phi'(z)} \ge  
\begin{cases}
\dfrac{\dist(\Gamma_1, \Gamma_4)}{600000\log R},\quad 
&1 < R < e^{\pi/4}; \\ 
(\diam K)/6 ,\quad & R \ge e^{\pi/4}
\end{cases}
\end{equation}
and
\begin{equation}\label{extderest2}
\abs{\Phi'(z)} \le 
\begin{cases}
5\cdot 10^9 \,\dfrac{\min(\diam \Gamma_2, \diam \Gamma_3)}{\log R},\quad & 1 < R < e^{2\pi}; \\ 
\diam K, \quad & R \ge e^{2\pi}.
\end{cases}
\end{equation}
\end{lemma}

\begin{proof} The conformal map $\Phi$  has the asymptotic behavior $\Phi(z) = c/z + O(1)$ as $z\to\infty$, where $|c|$ is the logarithmic capacity of $K$, denoted $\cp K$. For a compact connected set $K$, the logarithmic capacity is comparable to diameter: 
\begin{equation}\label{capdiam}
2\cp K\le \diam K\le 4\cp K,
\end{equation} 
see~\cite[\S 11.1]{Pomu}. A distortion theorem due to Loewner (see section IV.3 in~\cite{Golb} or Corollary~3.3 in~\cite{Pomu}) states that a univalent function $F\colon \C\setminus\overline{\DD}\to\C$, normalized by $F(z)/z\to 1$ as $z\to\infty$, satisfies
\begin{equation}\label{Sigmadist}
1-\frac{1}{\abs{z}^2} \le \abs{F'(z)}\le \frac{1}{1-1/\abs{z}^2}.
\end{equation}
Combining~\eqref{capdiam} with~\eqref{Sigmadist} yields
\[
\frac14 \diam  K\left(1-\frac{1}{\abs{z}^2}\right)
\le  \abs{\Phi'(z)} \le  \frac12 \frac{\diam K}{1-1/\abs{z}^2}
\]
which takes care of the second half of~\eqref{extderest1} and~\eqref{extderest2}.

Our next step is to prove the following distortion bounds, where $R = \abs{z}$ and $\rho = \dist(\Phi(z), K)$:
\begin{equation}\label{derdistext1}
\abs{\Phi'(z)} \ge \frac{\rho}{5 \log R} ,\quad  1 < R < e^{\pi/4};
\end{equation}
\begin{equation}\label{derdistext2}
\abs{\Phi'(z)} \le \frac{4 \rho}{\log R} ,\quad   R > 1. 
\end{equation}
Indeed, $(R^2-1)\abs{\Phi'(z)} \ge \rho$ is a consequence of the Schwarz-Pick lemma applied to $1/\Phi^{-1}$ in the disk $D(\Phi(z), \rho)$. The function $\log R/(R^2-1)$ is decreasing on the interval $(1, e^{\pi/4})$, hence is bounded below by its value at $e^{\pi/4}$, which is greater than $1/5$. The inequality~\eqref{derdistext1} follows.  To prove~\eqref{derdistext2}, apply the Koebe $1/4$ theorem to $\Phi$ in the disk $D(z, R-1)$. It yields 
$4\rho \ge (R-1)\abs{\Phi'(z)}\ge (\log R)\abs{\Phi'(z)}$ as claimed.   

\textit{Proof of ~\eqref{extderest1} for $1<R<e^{\pi/4}$.} 
From ~\eqref{distGamma} and ~\eqref{lowerharm1e} it follows that
\[
\dist(\Gamma_1, \Gamma_4) \le 
\dist(\zeta, \Gamma_1) + \dist(\zeta, \Gamma_4) \le 
8\csc^2 \left(\frac{1}{120}\right)\rho  \le 120000 \rho
\]
which in view of~\eqref{derdistext1}  implies 
\[
\abs{\Phi'(z)} \ge \frac{\rho}{5 \log R} \ge 
\frac{\dist(\Gamma_1, \Gamma_4)}{600000\log R}
\]
proving~\eqref{extderest1}.

\textit{Proof of ~\eqref{extderest2} for $1<R<e^{2\pi}$.} 
It follows from ~\eqref{diamGammaSimple} and ~\eqref{lowerharm2e} that
\[
\min(\diam \Gamma_2, \diam \Gamma_3) \ge  \frac{1}{32e^{2\pi}} \tan^2 \left(\frac{1}{256}\right)\rho 
\]
which in view of ~\eqref{derdistext2} implies  
\[
\abs{\Phi'(z)} \le \frac{4\rho}{\log R} \le 
5\cdot 10^9 \frac{\min(\diam \Gamma_2, \diam \Gamma_3)}{\log R}
\]
as claimed. 
\end{proof}

\section{Bi-Lipschitz extension of a centrally symmetric map}\label{globalsec}

\begin{proof}[Proof of Theorem~\ref{mainthm}] 
It suffices to work with a sense-preserving map $f\colon \T\to\C$. 
Our goal is to produce an extension $F\colon \C\to \C$ with the derivative bounds 
\begin{equation}\label{extbounds}
\norm{DF}\le 10^{27} L \quad \text{and} \quad  \norm{DF^{-1}} \le 10^{23} / \ell. 
\end{equation}
Indeed, the desired Lipschitz properties of both $F$ and $F^{-1}$ follow by integration along line segments.  Theorem~\ref{mainthmInt} provides an extension that satisfies~\eqref{extbounds} in $\DD\setminus\{0\}$. It remains to do the same in the exterior domain $\C\setminus\overline{\DD}$.  

Let $\Omega_e$ be the unbounded domain with the boundary $f(\T)$, and let $\Phi$ a conformal map of $\C\setminus\overline{\DD}$ onto $\Omega_e$. As in the proof of Theorem~\ref{mainthmInt}, we consider  
the induced boundary homeomorphism $\phi\colon \T\to f(\T)$ and define $\psi\colon \T\to\T$ by $\psi = f^{-1}\circ \phi$. Lemma~\ref{homeodist} provides an extension $\Psi\colon \DD\to\DD$ of $\psi$. Let 
$F = \Phi\circ r\circ \Psi^{-1}\circ r$ where $r(z) = 1/\bar z$ is the reflection in $\T$. It is easy to see that $F$ extends $f$. 

For $\zeta\in\C\setminus\overline{\DD} $ let $z = \Psi^{-1}(r(\zeta))$. By the chain rule, 
\[
\norm{DF(\zeta)} = \frac{\abs{\Phi'(r(z))} \norm{D\Psi(z)^{-1}}}{\abs{\zeta}^2\abs{z}^2} \quad 
\text{and} \quad 
\norm{DF(\zeta)^{-1}} = \abs{\zeta}^2\abs{z}^2\frac{\norm{D\Psi(z)}}{\abs{\Phi'(r(z))}}.
\]
According to~\eqref{modext},
\begin{equation}\label{modext2}
e^{-8\pi} \le \abs{\zeta}^2\abs{z}^2\le e^{8\pi}. 
\end{equation}
The claimed estimate for $\norm{DF(\zeta)^{-1}}$ for $1<\abs{\zeta}<e^{\pi/4}$  follows from the inequalities~\eqref{BAderest1}, ~\eqref{BLpropused},  ~\eqref{extderest1}, and~\eqref{modext2}:
\[
\norm{DF(\zeta)^{-1}} \le 600000 e^{12\pi}\pi /\ell \le 10^{23}/\ell.
\]
When $\abs{\zeta}\ge e^{\pi/4}$, we do not need ~\eqref{BLpropused} but use ~\eqref{twodisks} to obtain $\diam \partial\Omega \ge 2\ell$, which is used in~\eqref{extderest1}. Hence 
\[
\norm{DF(\zeta)^{-1}} \le 180e^{12\pi} /\ell  \le 10^{23}/\ell.
\]
Next, to estimate $\norm{DF(\zeta)}$ for $1<\abs{\zeta}<e^{2\pi}$ we use ~\eqref{BAderest2}, ~\eqref{BLpropused},  ~\eqref{extderest2}, and~\eqref{modext2}:
\[
\norm{DF(\zeta)} \le 5\cdot 10^9\cdot 4  e^{12\pi} L \le 10^{27}L. 
\]
The case $\abs{\zeta}\ge e^{2\pi}$ involves~\eqref{twodisks}, according to which $\diam \partial\Omega \le 2L$. Hence the combination of ~\eqref{BAderest2}, ~\eqref{modext2} and ~\eqref{extderest2} yields 
\[
\norm{DF(\zeta)} \le 2L  e^{8\pi}\cdot  16 e^{4\pi} \le 10^{27}L 
\]
completing the proof of Theorem~\ref{mainthm}. 
\end{proof}

\section{Symmetrization of a bi-Lipschitz embedding}\label{symmsec}

The \textit{winding map} $W\colon\C\to\C$ is defined in polar coordinates as $W(re^{i\theta}) = re^{2i\theta}$. 

\begin{definition}
Consider a homeomorphism $f\colon \T\to\C\setminus \{0\}$ such that $f(\T)$ separates $0$ from $\infty$. The \textit{winding symmetrization} of $f$ is a homeomorphism $g\colon \T\to\C\setminus \{0\}$ such that
\begin{equation}\label{fWWg}
f\circ W = W\circ g.
\end{equation}
\end{definition}

It is easy to see that $g$ is determined up to the sign, since $-g$ also satisfies~\eqref{fWWg}. 

To show the existence of $g$, observe that the winding number of $f$ about $0$ is $\pm 1$, which implies that the multivalued argument function 
$\arg f(e^{2it})$ increases by $\pm 4\pi$ as $t$ increases from $0$ to $2\pi$. Hence, we can define $g$ by  
\begin{equation}\label{symmhomeo}
g(e^{it}) = \exp\left(\frac{i}{2} \arg f(e^{2it})\right )
\end{equation}
Note that $g(-z)=-g(z)$ by construction, hence $g(-z)\ne g(z)$. This implies $g$ is injective, because if $z,\zeta \in \T$ are such that $\zeta \ne \pm z$, then 
\[ W(g(z)) = f(W(z))\ne f(W(\zeta)) = W(g(\zeta)).\]
The goal of this section is to determine what happens to the upper and lower Lipschitz constants of $f$ under symmetrization. 

The following example illustrates that there is an issue with the lower Lipschitz bound for $g$. 

\begin{example}\label{badcircle} Let $f(z) = z + 0.9$, which is obviously an isometry. Its symmetrization yields a map $g\colon\T\to\C$ such that $g(\pm i) = \pm 0.1$; thus, the lower Lipschitz constant of $g$ is at most $1/10$. The curve $g(\T)$ is shown in Figure~\ref{Symmetrizationisometry}.
\end{example}
 
\begin{figure}[ht]
\centering 
\includegraphics[width=0.5\textwidth]{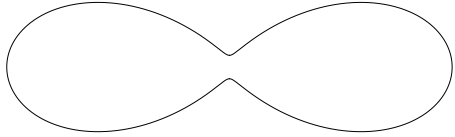}
\caption{Non-isometric symmetrization of isometry} \label{Symmetrizationisometry}
\end{figure} 

Figure~\ref{Symmetrizationisometry} also demonstrates that convexity may be lost in the process of winding symmetrization, and thus clarifies the difference between winding symmetrization  and \textit{central symmetrization}~\cite[p.~101]{Egg} which transforms closed convex curves into centrally symmetric closed convex curves. 

The issue with Example~\ref{badcircle} is that the curve $f(\T)$ is too close to $0$. This distance can be controlled with the following lemma, which is well-known but is proved here for completeness. 

\begin{lemma}\label{inradius} Let $f\colon \T\to\C$ be an $(L,\ell)$-bi-Lipschitz embedding. Denote by $R_I$ the inradius of the domain $\Omega$ bounded by $f(\T)$, that is the largest radius of a disk contained in $\Omega$. Then 
\begin{equation}\label{inradius1}
\ell\le R_I \le L.
\end{equation}
\end{lemma}
\begin{proof} By the Kirszbraun theorem~\cite[Theorem 1.34]{BB}, $f$ extends to an $L$-Lipschitz map $F\colon \C\to \C$. This extension need not be a homeomorphism, but we still have $\Omega\subset F(\DD)$ because $F_{\T} = f$ has nonzero degree with respect to each point of $\Omega$. It follows that every point of $\Omega$ is within distance $L$ of $F(\T) = \partial\Omega$, which means $R_I\le L$.  

Similarly, extending $f^{-1}$ to an $\ell^{-1}$-Lipschitz map $G\colon\C\to\C$ we find that the inradius of $G(\Omega)$ is at most $\ell^{-1}R_I$. Since $G(\Omega)\supset \DD$, the lower bound $R_I\ge \ell$ follows.
\end{proof}
 
\begin{proposition}\label{inradbound} Let $f\colon \T\to\C\setminus\{0\}$ be an $(L,\ell)$-bi-Lipschitz embedding. Define $r = \min_{\T}|f| > 0$. Then the symmetrized embedding $g$, defined by~\eqref{fWWg}, is $(\pi L, r\ell/(2\pi L))$-bi-Lipschitz. 
\end{proposition}

\begin{proof} Outside of $0$, the map $W$ is differentiable and its derivative matrix  has singular values $2$ and $1$. Hence $W$ is $2$-Lipschitz and locally invertible, with the inverse being $1$-Lipschitz. If the homeomorphism $f$ is $L$-Lipschitz, then its symmetrization $g$, which can be locally defined by $W^{-1}\circ f\circ W$, is locally $2L$-Lipschitz on $\T$.  It follows that $g$ is $2L$-Lipschitz with respect to the path metric on $\T$: 
\begin{equation}\label{pathL1}
\abs{g(z)-g(\zeta)}\le 2L\, \rho_{\T}(z,\zeta)
\end{equation}
where $\rho_{\T}(z,\zeta)$ is the infimum of lengths of curves joining $z$ to $\zeta$  and contained in $\T$. 
Since any two points $z,\zeta\in \T$ are joined by an arc of length at most $(\pi/2)\abs{z-\zeta}$, we have $\rho_{\T}(z,\zeta)\le (\pi/2)\abs{z-\zeta}$, hence
\[
\abs{g(z)-g(\zeta)}\le \pi L\abs{z-\zeta}.
\]
To prove the lower Lipschitz bound, fix $\zeta\in\T$. Note that $\abs{g(\zeta)}\ge r$ since $W$ preserves the absolute value. Consider two cases: 

\textit{Case 1.} $\rho_{\T}(z,\zeta)\ge \pi - r/(2L)$. Then $\rho_{\T}(-z, \zeta)\le r/(2L)$, which by inequality~\eqref{pathL1} implies $\abs{g(-z)-g(\zeta)} \le r$. Using the relation $g(-z) = -g(z)$ we obtain
\[
\abs{g(z)-g(\zeta)} = \abs{2g(\zeta) + g(-z)-g(\zeta)}
\ge \abs{2g(\zeta)} - \abs{g(-z)-g(\zeta)} \ge r \ge \frac{r}{2}\abs{z-\zeta}.
\]

\textit{Case 2.} $\rho_{\T}(z,\zeta) < \pi - r/(2L)$. An elementary geometric argument shows that the restriction of $W$ to an arc of $\T$ of size $\beta<\pi$ has lower Lipschitz constant $2\cos(\beta/2)$ with respect to the Euclidean metric. Therefore, 
\[
\abs{f(W(z))-f(W(\zeta))} \ge 2\ell\cos\left(\frac{\pi}{2} - \frac{r}{4L}\right) \abs{z-\zeta} = 2\ell\sin\left(\frac{r}{4L}\right)\abs{z-\zeta}.
\]
Since $f\circ W = W\circ g$ and $W$ is $2$-Lipschitz, it follows that 
\[
\abs{g(z)-g(\zeta)} \ge \frac12 \abs{f(W(z))-f(W(\zeta))} 
\ge \ell\sin\left(\frac{r}{4L}\right)\abs{z-\zeta}.
\]
The estimate $\sin x\ge 2x/\pi$, $0<x<\pi/2$, completes the proof. 
\end{proof}

\begin{remark}
The first part of the proof of Proposition~\ref{inradbound} can also be applied to $g^{-1}$, showing that $g^{-1} = W^{-1}\circ f^{-1}\circ W$ is $2L$-Lipschitz with respect to the path metric on $g(\T)$. However, this does not yield a bound on the Lipschitz constant of $g^{-1}$ in the Euclidean metric, since the shape of $g(\T)$ is unknown. 
\end{remark}

Combining Proposition~\ref{inradbound} with Lemma~\ref{inradius} we arrive at the following result.

\begin{corollary}\label{windBL} For every $(L,\ell)$-bi-Lipschitz embedding $f\colon \T\to\C$ there exists a point $w_0\in\C$ such that the winding symmetrization of $f-w_0$ is a $(\pi L, \ell^2/(2\pi L))$ bi-Lipschitz map.
\end{corollary}

The point $w_0$ can be taken to be an incenter of the domain bounded by $f(\T)$. 

\section{Conclusion}\label{concludesec}

\begin{proof}[Proof of Theorem~\ref{mainthmgen}] Given an $(L,\ell)$-bi-Lipschitz embedding $f\colon \T\to\C$, let $g$ be the winding symmetrization of $f-w_0$ as in Corollary~\ref{windBL}. Theorem~\ref{mainthmgen} provides 
its bi-Lipschitz extension $G\colon \C\to \C$ which is also centrally symmetric. Therefore there exists $F\colon \C\to\C$ such that $F\circ W = W\circ G$. Since the singular values of the derivative matrix $DW$ are $1$ and $2$, it follows that $\sup \norm{DF} \le 2\sup\norm{DG}$ and $\sup \norm{DF^{-1}} \le 2\sup\norm{DG^{-1}}$. Recalling the Lipschitz bounds of  Corollary~\ref{windBL} and Theorem~\ref{mainthm}, we arrive at 
\[
\norm{DF} \le 2\pi\cdot 10^{27}L \le 10^{28}L
\]
and
\[\norm{DF^{-1}}^{-1}\ge \frac12\cdot  10^{-23}\frac{\ell^2}{2\pi L} \ge 10^{-25}\frac{\ell^2}{L}.
\]
One of the maps $F$ and $-F$ provides the desired extension of $f$. 
\end{proof}

The source of   nonlinearity in Theorem~\ref{mainthmgen} is the symmetrization process of section~\ref{symmsec}.  
It is thus natural to seek an form of Corollary~\ref{windBL} with a linear bound for the lower Lipschitz constant. The following example shows that such an improvement will require a better way of choosing the center point $w_0$ for symmetrization. 

\begin{example} Let $f\colon \T\to\C$ be the map described by Figure~\ref{BeforeFigure}, where $f(a)=A$, $f(b)=B$ and both boundary curves $AB$ and $BA$ are traced counterclockwise with constant speed. The Euclidean distances $\abs{a-b}$ and $\abs{A-B}$ are equal to a small parameter $\epsilon$. The speed at which $AB$ is traced is about $1/\epsilon$, while the speed of $BA$ is about $1$. One can see that $f$ is bi-Lipschitz with respect to the Euclidean metric with constants approximately $(1/\epsilon, 1)$. 
\begin{figure}[ht]
\centering 
\includegraphics[width=0.5\textwidth]{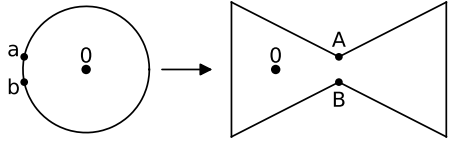}
\caption{Bi-Lipschitz constants about $(1/\epsilon, 1)$} \label{BeforeFigure}
\end{figure} 

The map obtained after winding symmetrization is shown on Figure~\ref{AfterFigure}.  Both distances $\abs{A_1-B_1}$ and $\abs{A_2-B_2}$ are of order $\epsilon$. However, $\abs{a_1-b_1}$ and $\abs{a_2-b_2}$ are approximately $2$. Thus, the lower Lipschitz constant of the symmetrized map decays with $\epsilon$. The ratio of upper and lower Lipschitz constants gets squared in the process of winding  symmetrization, as is does in Corollary~\ref{windBL}.
\begin{figure}[ht]
\centering 
\includegraphics[width=0.95\textwidth]{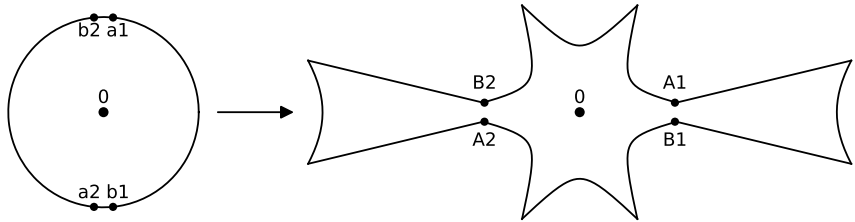}
\caption{Bi-Lipschitz constants about $(1/\epsilon, \epsilon)$} \label{AfterFigure}
\end{figure} 
\end{example}

However, if the map on Figure~\ref{BeforeFigure} was translated so that $0$ is in the center of the right half of the bowtie, the winding symmetrization would not incur a nonlinear growth of distortion. This motivates the following question.

\begin{question}\label{optimistic} Is there a universal constant $C$ such that for every $(L,\ell)$-bi-Lipschitz embedding $f\colon \T\to\C$ there exists a point $w_0\in\C$ such that the winding symmetrization of $f-w_0$ is a $(C L, \ell/C)$ bi-Lipschitz map? 
\end{question}

A positive answer to Question~\ref{optimistic} would provide linear distortion bounds in Theorem~\ref{mainthmgen}, thus answering the question of Daneri and Pratelli~\cite{DP}.

\bibliographystyle{amsplain}

\end{document}